\documentclass{article}
\usepackage{amsmath}
\newtheorem{theorem}{Theorem}[section]
\newtheorem{lemma}[theorem]{Lemma}
\newtheorem{proposition}[theorem]{Proposition}
\newtheorem{corollary}[theorem]{Corollary}
\newtheorem{conjecture}[theorem]{Conjecture}
\newtheorem{definition}[theorem]{Definition}

\newenvironment{proof}[1][\it{Proof.}]{\begin{trivlist}
\item[\hskip \labelsep {\bfseries #1}]}{\end{trivlist}}

\newcommand{\qed}{\nobreak \ifvmode \relax \else
      \ifdim\lastskip<1.5em \hskip-\lastskip\
      \hskip1.5em plus0em minus0.5em \fi \nobreak
      \vrule height0.75em width0.5em depth0.25em\fi}

\title{ Integral dimension of a noetherian ring}
\author{{   Caijun  Zhou }
\\
{\small {\it Department of  Mathematics, Shanghai Normal University,}}\\
{\small {\it Shanghai, 200234, China}}}
\date{}

\begin{document}
\maketitle
\section {Introduction}

\ \ \ \  Throughout this paper we always use $R$ to denote a commutative
noetherian ring with an identity. The main purpose of the paper is to introduce a new notion, integral dimension,   for noetherian rings.

  Let $J$ be an ideal of $R$. An
element $x\in J$ is said to be integral over $J$, if there is an equation
of the form
$$x^n+a_1x^{n-1}+\cdots+a_n=0$$
where $a_i\in J^i$ for $1\leq i\leq n$. It is not difficult to show that the
set of all elements which are integral over $J$ is an ideal of $R$. We will use $\overline{J}$ to denote this ideal,
and call it the integral closure of the ideal $J$. An ideal $I$ with $I\subseteq \overline{J}$ will be said to be integral
over $J$.

 One of the main interests on $\overline{J}$ is when $\overline{J^n}$ is contained in a power of $J$. Let us recall several
 remarkable results established in the past years. Rees [Re] showed that if $R$ is an analytically unramified local
 ring and $J$ an ideal of $R$, then there is an integer $k$, depending on $J$, such that $\overline {J^n}\subseteq {J^{n-k}}$ for $n\geq k$.

 Let $R$ be the ring of convergent power series in $n$ variables $Z_1, \cdots, Z_n$ over the field of complex numbers  $\textbf{C}$.
 Let $f\in R$ be a non unit element and $J=(\partial f/\partial{Z_1}, \cdots, \partial f/\partial{Z_n})$ be the Jacobian
 ideal of $f$. It can be proved [cf. HS] that $f$ is integral over $J$. Brian\c{c}on and Skoda [BS] showed by analytic method that $f^n\in J$
 , and this is known as  the Brian\c{c}on-Skoda theorem.

 Lipman and Sathaye [LS] extended the theorem of Brian\c{c}on-Skoda to arbitrary  regular rings by purely algebraic method. They proved if $R$  ia a   regular ring of  dimension $d>0$, then for all ideals $J$ of $R$,   $\overline {J^n}\subseteq {J^{n-d+1}}$ for $n\geq d-1$.

 Later on, Huneke [ Hu, Theorem 4.13] proved a far reaching result for a broad class of noetherian rings. For example,  He obtained that, for any reduced excellent  local ring $R$, there exists an integer $k\geq 0$, such that for all ideals $J$ of $R$,  $\overline {J^n}\subseteq {J^{n-k}}$ for $n\geq k$. The result of Huneke is known as  the uniform  Brian\c{c}on-Skoda theorem. Recently,  Andersson and Wulcan  [AW] presented global effective versions of the   Brian\c{c}on-Skoda-Huneke theorem by analytic method.

 In this paper, instead of considering the integral closure of the power  ideal   $\overline{J^n}$, we will study the power of the integral closure of an ideal ${\overline J}^n$.   We  prove in the next section for any ideal $J$ of $R$,
 $$J^n\subseteq {\overline J}^n\subseteq \overline{J^n}.$$
 Hence the comparison of ${\overline J}^n$ with the power of $J$ may be easier than that of $\overline{J^n}$.
  Motivated by this, we introduce the  following notion of integral dimension for  noetherian rings.

 \begin{definition}
Let $R$ be a  noetherian ring.  If there exists an integer $k\geq 0$ such that  for all ideals  $J$ of $R$ and for $n\geq k$
 $${\overline J}^n\subseteq {J^{n-k}}, \eqno(1.1)$$
then we define the integral dimension of $R$, denoted by $i(R)$,  to be the least integer $k$  such that (1.1) holds for all $J$ and all $n\geq k$.  If  such $k$ does not exist, then we set $i(R)=\infty$.
\end{definition}

It is obvious that $i(R)=k< \infty$ if and only if $k$ is the least integer such that for every pair of ideals $J\subseteq I$, where $I$ is
integral over $J$, $$I^n\subseteq J^{n-k}$$ for $n\geq k$. Moreover, we have $Q^{k+1}=0$ for the nilradical  $Q$ of $R$ because it  is integral
over the zero ideal. Clearly, if $R$ is an artinian ring, then $i(R)\leq l(R)$, where $l(R)$ denotes the length of $R$.

It follows easily from  the Lipman-Sathaye theorem, $i(R)\leq d-1$ for every
regular ring of  dimension $d>0$. The above mentioned result of Huneke implies that if $R$ is a reduced excellent local ring, then $i(R)<\infty$.

In this paper, we will prove  some basic facts about the behavior of $i(R)$ under  several operations on $R$. The results show that
$i(R)$ behaves well under operations such as localization and completion.

We will give a   lower bound for $i(R)$ in terms of the dimension $d$ of $R$. Explicitly, we will show
 $i(R)\geq d-1$.  In particular, it follows from Lipman-Sathaye theorem, if $R$ is a regular ring of dimension $d>0$,  then $i(R)= d-1$.  We can prove  a  noetherian ring $R$ with $i(R)=0$  if and only if $R$ is a regular ring of dimension $d\leq 1$.

The main goal  of the paper is to show that  $i(R)<\infty$ for  a large class of noetherian  rings. Recall that a $S$-algebra $R$ is said to be essentially of finite type over a ring $S$, if $R$ is a localization of a finitely generated $S$-algebra. Our main  result  of the paper states as follows:

\vspace* {0.2cm} \noindent

{\bf Theorem 3.6}\ \ {\it  Let $R$ be a noetherian ring of finite  dimension. If $R$ satisfies one of the following conditions, then $i(R)<\infty$.

{\rm{(i)}}   $R$ is  essentially of finite type over a  local ring $S$.

{\rm{(ii)}} $R$ is a ring of characteristic $p$, and $R$ is module finite over $R^p$.

{\rm{(iii)}} $R$ is essentially of finite type over the ring  of integer numbers $\textbf{Z}$.
 }

\

In particular, one can conclude  from  Theorem 3.6,  $i(R)<\infty$ for any noetherian local ring $R$.  It implies that the notion of the integral dimension  becomes a well-defined notion and  has  concrete  meaning  in the class of local rings.  We do not know  whether  $i(R)$ is finite  for every noetherian ring $R$. It seems very difficult to give an answer to the question in general.  However, we can prove (see, Corollary 2.11)  that $\text{dim}(R)<\infty$ if   $i(R)<\infty$.  We  make the following conjecture.

\begin{conjecture} Let $R$ be a noetherian  ring of finite dimension. Then $i(R)<\infty.$

\end{conjecture}

Huneke [Hu] observed that the uniform  Brian\c{c}on-Skoda property plays an  important role in the study of  the  uniform Artin-Rees property of a noetherian ring.  As  pointed out  in Theorem 3.2, the finiteness property  of integral dimension  is also  useful in proving a ring with such uniform  Artin-Rees property.  The proof of our main result Theorem 3.6 depends heavily  on the  uniform Artin-Rees theorem of Huneke [Hu].
The main technique  of Huneke came from the paper of Lipman and Sathaye [LS] in characteristic 0 and the theory of tight closure of [HH1] [HH2] in characteristic $p$. It can not be applied directly to proving a result more general than Theorem 3.6.

Even for the local case, it would be very interesting to reveal  more about the relationship between $i(R)$ and  $i(R/aR)$ for an element $a$ in  $R$ with $\text{dim}(R/aR)=\text{dim}(R)-1$. We guess that  they may be related to each other by something like  the multiplicities, $e(R)$ and $e(R/(a))$ ,  of the maximal ideals of $R$ and $R/(a)$. The answer is not known even for   $i(R)$ and $i(R[X])$,   where $R[X]$ is the polynomial ring of $R$ in the variable $X$. Hence  there are a lot of questions about integral dimension remained for further studying.

 \vspace* {0.3cm}

\section{Basic properties}

\ \ \ \ \ \   In this section,  we will recall some basic facts about the integral closure of an ideal, and then presents some properties of the integral dimension $i(R)$. Some of the results  are known, one can find in [cf. HS], we will prove them for  convenience of readers.

 Let $Q$ be the  nilradical of $R$ and  $R_{red}$ be the ring of $R$ modulo $Q$. Since there is an integer $n$ such that $Q^n=0$, it is easy to see that $Q$ is contained in $\overline J$ for every ideal $J$ of $R$, and an element $x$ is integral over $J$ if and only if $\bar x$, the image of $x$ in $R_{red}$,  is integral over $JR_{red}$. Thus we have:

\begin{proposition} Let $J$ be an ideal of $R$. Then  $\overline {JR_{red}} = \overline {J}R_{red}$.

\end{proposition}

Integral dependence  behaves well under localization:

\begin{proposition} Let $J$ be an ideal of $R$. Then for any multiplicatively
closed subset $T$ of $R$, $T^{-1} \overline J = \overline {T^{-1}J}.$

\end{proposition}

\begin{proof} Clearly, $T^{-1} \overline J \subseteq  \overline {T^{-1}J}$ by the definition of integral dependence. Let $a\in \overline {T^{-1}J}$. By the definition of integral dependence again, there is an equation of the form
$$ a^n+ a_1a^{n-1}+\cdots+ a_n=0$$
in the ring $T^{-1}R$, where $a_i\in T^{-1}J^i$ for $1\leq i\leq n$. Choose an element $b\in T$ such that $ba\in R$ and $ba_i\in J^i$ for $1\leq i\leq n$. Hence we have
$$ {(ba)}^n+ a_1b{(ba)}^{n-1}+\cdots+ b^na_n=0$$
is an equation with coefficients in $R$, and there is an element $b'\in T$ such that
 $$b'({(ba)}^n+ a_1b{(ba)}^{n-1}+\cdots+ b^na_n)=0$$
 in $R$. It shows that
$$ {(b'ba)}^n+ a_1b'b{(b'ba)}^{n-1}+\cdots+ {(b'b)}^na_n=0$$
is the desired equation which implies $b'ba$ is integral over $J$. So $a\in T^{-1} \overline J $, and $\overline {T^{-1}J}\subseteq T^{-1} \overline J $. This proves the proposition.

\end{proof}

\begin{corollary}
Let $T$ be any multiplicatively closed subset of $R$. Then $ i(T^{-1}R)\leq i(R)$.
\end{corollary}

\begin{proof} If $i(R)=\infty$, there is nothing to prove. Now, assume that $i(R)=k$ is finite. For any ideal $K$ of $T^{-1}R$, there exists an ideal $J$ of $R$ such that $K=T^{-1}J$. By Proposition 2.2, $\overline K = T^{-1}{\overline J}$. Note that ${\overline J}^n\subseteq J^{n-k}$ for $n\geq k$. Thus,  for $n\geq k$,  we have
$${\overline K}^n={(T^{-1}{\overline J})}^n = T^{-1}{\overline {J}}^n\subseteq {T^{-1}{{ J}}^{n-k}}={(T^{- 1}J)}^{n-k}=K^{n-k}$$
and it follows that  $i(T^{-1}R)\leq k$.

\end{proof}

 It is well-known the  dimension of $R$ can be computed as follows:
 $$\text{dim}(R)= sup \{\ \text{dim}(R_\texttt{m})\ |\ \texttt{m} \text { a maximal ideal of R}\},$$
  where $R_\texttt{m}$ denotes the localization of $R$ at the maximal ideal $\texttt{m}$. Similarly, we have the following computation for integral dimension.

\begin{corollary} For a noetherian ring $R$, $i(R) = sup \{\ i(R_\texttt{m})\ |\ \texttt{m}\text { a maximal ideal of R}\}.$
\end{corollary}

\begin{proof} Let us  set
$$t= sup \{\ i(R_\texttt{m}) \ |\ \texttt{m}{\text { a maximal ideal of }}R\}.$$

By Corollary 2.3, we have $i(R)\geq t.$ On the other hand, if $t=\infty$, then
$i(R)=\infty$  and thus $i(R)=t$ holds in this case.

 Now, assume that $t< \infty$. For any ideal $I$ of $R$,  we have $i(R_\texttt{m})\leq t$
and
$$(\overline{IR_\texttt{m}})^n \subseteq (IR_\texttt{m})^{n-t}$$
for all for $n\geq t $ and for any maximal ideal $\texttt{m}$ of $R$. By Proposition 2.2, $\overline{IR_\texttt{m}}=\overline{I}R_\texttt{m}$. It implies
$${\overline{I}}^nR_\texttt{m} \subseteq I^{n-t}R_\texttt{m}$$
for all maximal ideals $\texttt{m}$. Hence  ${\overline{I}}^n \subseteq I^{n-t}$ for $n\geq t,$  and consequently $i(R)\leq t$. This proves the corollary.
\end{proof}

Now, we recall  a useful tool, the reduction of an ideal,  for  studying the integral closure of an ideal. An ideal $J\subseteq I$ is said to be a reduction of the ideal $ I$,  if there exists an integer $k$ such that $I^{k+l}=JI^k$.  One important result about  reduction  is the following proposition.

\begin{proposition} Let $R$ be a noetherian ring and  $J\subseteq I$  a pair of ideals of $R$. Then the following conditions are equivalent.

{\rm{(i)}} $J$ is a reduction of $I$.

{\rm{(ii)}} $I$ is integral over $J$.

\end{proposition}

\begin{proof}(i) $\Rightarrow$ (ii) By assumption, there is an integer $k$ such that $I^{k+l}=JI^k$.  As $R$ is a noetherian ring, the ideal $I^k$ is finitely generated. Set $I^k=(a_1, a_2, \cdots, a_n)$. For $a\in I$ and $i (1\leq i\leq n )$, write $a a_i=\sum_{j=1}^{n}a_{ij}a_j$ for some $a_{ij}\in J$. Let $ A=(a_{ij})$ be the matrix $(\delta_{ij}a - a_{ij})$, where $\delta_{ij}$ is the Kronecker delta function. By Cramer's Rule, $\text{det}(A)I^k=0$. In particular, $\text{det}(A)a^k=0$,  and an expansion of
$\text{det}(A)a^k=0$ yields the desired equation of integral dependence of $a$ over $J$. Hence $I$ is integral over $J$.

(ii) $\Rightarrow$ (i) Suppose that $I$ is integral over $J$. Write $I=(b_1, b_2,\cdots, b_s)$. For each $i, 1\leq i\leq s$, there is an equation of the form
 $$b_i^{n_i}+c_{i1}b_i^{n_i-1}+\cdots+c_{in_i}=0$$
where $ c_{ij}\in I^j$ for $1\leq j\leq n_i$. Thus it implies $b_i^{n_i}\in JI^{n_i-1}$. Put $k=\sum_{j=1}^s n_j +1$. It is clear that if $r_1+r_2+\cdots+r_s=k$, every element of the
form $b_1^{r_1}b_2^{r_2 }\cdots b_s^{r_s}$ lies in $JI^{k-1}$, and this proves that $I^k=JI^{k-1}$.
\end{proof}

As an easy consequence of the proposition, we have the following result mentioned in the last section.

\begin{corollary} Let $J$ be an ideal of $R$. Then for any integer $n$, $J^n\subseteq {\overline J}^n\subseteq \overline{J^n}.$

\end{corollary}

\begin{proof}  It follows from Proposition 2.5, $J$ is a reduction of ${\overline J}$. Thus there is an integer $k$ such that ${\overline J}^{k+1}=J {\overline J}^k$. For an integer $n$, we have ${({\overline J}^n)}^{k+1}=J^n{({\overline J}^n)}^k.$ By Proposition 2.5 again, ${\overline J}^n$ is integral over $J^n$, and consequently $$J^n\subseteq {\overline J}^n\subseteq \overline{J^n}.$$

\end{proof}

Let $\varphi: R\rightarrow S$ be a morphism of noetherian rings. The property of persistence of integral closure states that $ \varphi(\overline J)\subseteq \overline {\varphi(J)}$ for any ideal $J$ of $R$. This follows as by applying $\varphi$ to an equation of integral dependence of an element $a$ over $J$ to obtain an equation of integral dependence of $\varphi(a)$ over the ideal $\varphi(J)S$.

\begin{proposition} Let $\varphi: R\rightarrow S$ be a faithfully flat  morphism of noetherian rings. Then $i(R)\leq i(S).$

\end{proposition}

\begin{proof} If $i(S)=\infty$, there is nothing to prove. Now, we assume $i(S)=k$ is finite. Let $J$ be an ideal of $R$.  Since $i(S)=k$, it shows
 ${\overline {\varphi(J)}}^n \subseteq {\varphi(J)}^{n-k}$ for $n\geq k$. By persistence of integral closure, we have $\varphi(\overline J)\subseteq \overline {\varphi(J)}$. So ${ {{\varphi(\overline {J})}}}^n \subseteq {\varphi(J)}^{n-k}.$ It implies ${ {{\varphi{(\overline {J}}^n)}}} \subseteq {\varphi(J^{n-k})}.$ Note that $\varphi$ is a faithfully flat  morphism of noetherian rings,  we conclude ${\overline J}^n\subseteq J^{n-k}$ for $n\geq k$, and this proves $i(R)\leq i(S)$.

\end{proof}

If $R$ is a local ring with the unique maximal ideal $\texttt{m}$, we use $\hat{R}$ to denote the $\texttt{m}$-adic completion of $R$.  Since $\hat R$ is faithfully flat over $R$,  $i(R)\leq i(\hat R)$ by Proposition 2.7. We prove that this inequality is indeed an equality.

\begin{proposition} Let $(R, \texttt{m})$ be a noetherian local ring. Then  $i(R)= i(\hat R).$

\end{proposition}
\begin{proof} It suffices to prove $i(\hat R)\leq i(R)$ by Proposition 2.7. If $i(R)=\infty$,  there is nothing to prove. In fact, such case cannot happen, we will  prove in the next section, $i(R)$ is always a finite number. Now, put $i(R)=k$,  and naturally regard $R$ as a sub-ring of $\hat R$. Let $K=(b_1, b_2,\cdots, b_r)$ be an arbitrary ideal of $\hat R$.

If $K$ is a $\texttt{m}\hat R$-primary ideal, then there exists an integer $t$ such that  $\texttt{m}^t\hat R\subseteq K$. By the definition of completion of a ring, we can choose elements $a_i\in R$ such that $b_i-a_i$ lies in $m^{t+ 1}\hat R$ for $1\leq i\leq r$. Let $J $ be the ideal of $R$ which is generated by $a_1, a_2, \cdots, a_r$.  Clearly, $K \subseteq J\hat R + \texttt{m}K$. So it follows that $K=J\hat R$ by Nakayama lemma. Since the integral closure $\overline K$ of $K$ is also a $\texttt{m}\hat R$-primary ideal, there is an ideal $I$ of $R$, replacing $I$ by $I + J$, we may assume $J\subseteq I$, such that $\overline K = I\hat R$. By Proposition 2.5, $J\hat R$ is a reduction of $I\hat R$, and there is an integer $n_0$ such that $I^{n_0 +1}\hat R = JI^{n_0 }\hat R$. As $\hat R$ is faithfully flat over $R$, we obtain $I^{n_0 +1}= JI^{n_0}$.  Hence for $n\geq k$, $I^n\subseteq J^{n-k}$, and consequently
$${\overline K}^n = I^n\hat R\subseteq J^{n-k}\hat R = (J\hat R)^{n-k}=K^{n-k}.$$

For an arbitrary ideal $K$ of $\hat R$, we have for $n\geq k$ and all positive integer $i$
$${\overline K}^n\subseteq {\overline {K+\texttt{m}^i\hat{R}}}^n\subseteq ({K+\texttt{m}^i\hat{R}})^{n-k}\subseteq K^{n-k}+ \texttt{m}^i\hat R$$
by the case we have just proved. So ${\overline K}^n\subseteq K^{n-k}$ for $n\geq k$  by  Nakayama lemma again. This proves $i(\hat R)\leq i(R)$, and the proof of the proposition is complete.

\end{proof}

Now, we turn to giving  a lower bound for $i(R)$. We will show that if $R$ is of finite dimension $d>0$, then $i(R)$ is at least $d-1$. Let us begin with recalling Monomial Conjecture of Hochster which asserts that given any system of  parameters $x_1, x_2, \cdots, x_d$ of a $d$-dimensional local ring $R$,
then for all $n>1$
$$(x_1x_ 2\cdots x_ d)^n\notin (x_1^{n+1}, x_2^{n+1}, \cdots, x_d^{n+1}).$$
Hochster proved the conjecture in the equicharacteristic case [Ho1] [Ho2]. In mixed cases,  the conjecture remains undetermined, despite much effort. However,  Hochster [Ho1] pointed out,  for any system of  parameters $x_1, x_2, \cdots, x_d$ of $R$, there exists an integer $t_0$, such that  for all $t\geq t_0 $,  $x_1^t, x_2^t, \cdots, x_d^t$ satisfying Monomial Conjecture. In particular, we have:

\begin{lemma} If $R$ is a noetherian local ring of dimension $d>0$, then there exists a system of  parameters $x_1, x_2, \cdots, x_d$ of $R$
such that for all $n>0$
$$(x_1x_ 2\cdots x_ d)^n\notin (x_1^{n+1}, x_2^{n+1}, \cdots, x_d^{n+1}).$$

\end{lemma}

By means of Lemma 2.9, we can prove the following interesting result.

\begin{proposition}
Let $R$ be a noetherian ring  dimension $d>0$. Then $i(R)\geq d-1$.
\end{proposition}

\begin{proof} Choose a maximal ideal $\texttt{m}$ of $R$ such that $\text{ht}(\texttt{m})=d$. It is clear, $\text{dim} (R_\texttt{m}) = d$. By Corollary 2.3, $i(R)\geq i(R_\texttt{m})$. Replacing $R$ by $R_\texttt{m}$, we may assume $R$ is a local ring with the maximal ideal $\texttt{m}$.

If $d=1$, it follows from the definition of integral dimension that $i(R)\geq 0$, and thus the conclusion is true in this case.

Now we assume $d\geq 2$. By Lemma 2.9, there  is a system of  parameters $x_1, x_2, \cdots, x_d$ of $R$ such that the element
 $(x_1x_ 2\cdots x_ d)^n$ is not contained in the ideal $(x_1^{n+1}, x_2^{n+1}, \cdots, x_d^{n+1})$ for any $n>0$.

Let us consider the two ideals
 $I=(x_1^d, x_1^{d-1}x_d, x_2^d, x_2^{d-1}x_d, \cdots, x_{d-1}^d, x_{d-1}^{d-1}x_d, x_d^d)$
and  $J=(x_1^d, x_2^d, \cdots, x_d^d)$ of $R$. Since for each $i \ (1\leq i \leq {d-1})$
$$(x_i^{d-1}x_d)^d = (x_i^d)^{d-1}x_d^d\in (x_i^d, x_d^d)(x_i^d, x_i^{d-1}x_d, x_d^d)^{d-1},$$
it yields  that
$$(x_i^d, x_i^{d-1}x_d, x_d^d)^d = (x_i^d, x_d^d)(x_i^d, x_i^{d-1}x_d, x_d^d)^{d-1}.$$
By Proposition 2.5, $x_i^{d-1}x_d$ is integral over the ideal $(x_i^d, x_d^d)$, and thus $x_i^{d-1}x_d$
is integral over $J$ for $i \ (1\leq i\leq d-1)$. It shows $I$ is generated by elements which are integral over $J$, and thus $I$ is integral over $J$.

Suppose that $i(R)<d-1$. We must have $I^{d-1}\subseteq J$ by the definition of integral  dimension. In particular, it implies the multiplication of the following $d-1$ elements
 $$x_1^{d-1}x_d, \cdots, x_{d-1}^{d-1}x_d$$  lies in $J$, i.e.
$$(x_1x_2\cdots x_d)^{d-1}\in (x_1^d, x_2^d, \cdots, x_d^d),$$
and this contradicts  the choices of  $x_1, x_2, \cdots, x_d$. Therefore $i(R)\geq d-1$, and this proves the conclusion.

\end{proof}

As an immediate consequence of Corollary 2.4 and  Proposition 2.10, we have:

\begin{corollary} Let  $R$ be a noetherian ring. If $i(R) < \infty $, then $\text {dim} (R)< \infty.$
\end{corollary}

Motivated by Proposition 2.10, we will say  a noetherian local ring $R$ of dimension $d$ is of   minimal integral dimension if $i(R)= d-1$.
It follows from  the Lipman-Sathaye theorem [LS], $i(R)\leq d-1$ for every
regular ring of  dimension $d>0$. So by Proposition 2.10 , we
have:

\begin{corollary} Let $R$ be a regular local ring of dimension $d>0$. Then $R$ is of minimal integral dimension, i.e. $i(R)=d-1$.
\end{corollary}

   In the rest of the section, we consider the local ring $R$ with $i(R)=0$. Clearly every ideal of $R$ is integrally closed for such rings.  By Proposition 2.10, we have $\text{dim} R\leq 1$.  If $R$ is an artinian local ring,  then every ideal $I$ of $R$ with $I\neq R$ is integral over the zero ideal $0$.  So $I=0$, and $R$  must be a field. For $\text{dim} R=1$, we have

\begin{proposition} Let $(R, \texttt{m})$ be a noetherian local ring with  $i(R)=0$. If $R$ is not a field, then $R$ is a one dimensional regular local ring.
\end{proposition}

\begin{proof} As we point out as above, under the assumption of $R$, $R$ must be a local ring of dimension 1. Since every ideal of $R$   is integrally closed, it follows that $R$ is  reduced  because the nilradical of $R$ is integral over the zero ideal.

Now, we show that $R$ is a domain. It suffices to prove that $R$ has only one minimal prime ideal. Let $P_1, P_2, \cdots, P_r$ are all the minimal
prime ideals of $R$.

If $r\geq 2$, we can choose elements $a, b$ such that
$$a\in P_1, a\notin \cap_{i=2}^rP_i \ \text{and} \ b\in \cap_{i=2}^rP_i, b\notin P_1.$$
It is obvious $ab=0, (a+b)\notin P_1$ and $ a^2-(a+b)a=0.$  It shows $a$ is  integral over the ideal $(a+b)R$,  and thus  $a\in (a+b)R$ because of   $\overline{(a+b)R}=(a+b)R$.

Write $ a=(a+b)c$
for some $c\in R$. If $c\notin \texttt{m}$, then $c$ is a unit of $R$. So $$a+b=c^{-1}a\in P_1,$$
 and this contradicts  the choice of $a+b$. Hence $c\in \texttt{m}$, and it implies $a(1-c)=bc$. Thus $$a\in \cap_{i=2}^rP_i$$
 because $1-c$ is a unit of $R$, which  also contradicts the choice of $a$. Hence $r=1$ and $R$ is a domain.

Let $Q(R)$ be the field of the fractions of $R$. In order to prove that $R$ is a regular local ring, it suffices to show that $R$ is integrally closed in $Q(R)$ [Ma]. Let $\frac{a}{b}$ be an element of $Q(R)$ such that $\frac{a}{b}$ is integral over $R$, where $a, b\in R$.  Hence there exists an equation of the form
$$ (\frac{a}{b})^n +a_1(\frac{a}{b})^{n-1} + \cdots +a_n=0$$
where $a_i\in R$ for $(1\leq i\leq n)$. It yields that
$$ a^n +a_1ba^{n-1} + \cdots +b^na_n=0.$$
It implies  $a$ is integral over the ideal $bR$. Since $bR$ is integrally closed, it shows $a\in bR$, and thus $\frac{a}{b}\in R$. Therefore $R$ is a integrally closed domain, and the proof of the proposition is complete.

\end{proof}

 The result of Proposition 2.13 means that a one dimensional noetherian local ring with minimal integral dimension is a regular local ring.  It would be very interesting to study the properties of local rings with nonzero minimal integral dimension. We conjecture:

\begin{conjecture} Let $(R, m)$ be a reduced local ring with minimal integral dimension  $i(R)>0$. Then $R$ is a regular local ring.

\end{conjecture}

\vspace* {0.3cm}

\section{Main result}

\ \ \ \ \ \ In this section, we will  present several sufficient conditions for a ring $R$ having finite integral
dimension.  Let us begin with recalling the  notion of uniform Artin-Rees property   studied  by Huneke [Hu].

Let $R$ be a noetherian ring, we say that $R$ has the uniform Artin-Rees property, if for every pair $N\subseteq M$ of finitely generated $R$-modules, there exists a number $k$, depending on $N, M$, such that for all ideal $I$ of $R$ and all $n\geq k$
$$I^nM\cap N\subseteq I^{n-k}N.$$

First of all, we observe that if $R$ has the  uniform Artin-Rees property,  then we can reduce the question when $i(R)$ is finite to the
the reduced case when $i(R_{red})<\infty$.

\begin{proposition}
Let $R$ be a noetherian ring. Then

{\rm{(i)}} $i(R_{red})\leq i(R)$.

{\rm{(ii)}} If $i(R_{red})<\infty$, then $i(R) < \infty, $ provided  that $R$ has the uniform Artin-Rees property.

\end{proposition}

\begin{proof} (i)  If $i(R)=\infty$,  we must have $i(R_{red})\leq \infty$ by (i), and the conclusion holds in this case.

Now we assume $i(R)$ is a finite number and put $i(R)=k$.
For any ideal $K$ of $R_{red}$, there is an ideal $I$ with $Q\subseteq I$ such that $K=IR_{red}$. Note that  ${\overline I}^n \subseteq I^{n-k}$ for $n\geq k$. By Proposition 2.1, $\overline K ={\overline I}R_{red}$. So
$${\overline K}^n={\overline I}^nR_{red}\subseteq I^{n-k}R_{red}=K^{n-k}$$
for $n\geq k$,  this shows $i(R_{red})\leq k$, and the proof of (i) is complete.

 (ii)  Let $Q$ be the nilradical of $R$. If $Q=0$, then $R=R_{red}$, and there is nothing to prove. Now, we assume $Q\neq 0$.  Let $n_0$ be the positive integer such that $Q^{n_0}\neq 0$  and  $Q^{n_0+1}=0.$  Since $R$ has the uniform Artin-Rees property, for each $i  (1\leq i\leq n_0)$, there exists integer $k_i>0$ such that for all ideals $I$ of $R$ and all $n\geq k_i$
$$I^n\cap Q^i\subseteq I^{n-k_i}Q^i. \eqno (3.1)$$

Put $k=i(R_{red})$. By Proposition  2.1, for any ideal $I$ of $R$, $\overline {IR_{red}} = \overline {I}R_{red}$. It yields
$${\overline {I}}^nR_{red}\subseteq I^{n-k}R_{red}$$
for $n\geq k$. Equivalently, we have for $n\geq k$
$${\overline {I}}^n\subseteq I^{n-k} + Q.  \eqno (3.2)$$

Since $I^{n-k}\subseteq {\overline I}^{n-k}$, it shows
$${\overline {I}}^n\subseteq I^{n-k} + {\overline {I}}^{n-k}\cap Q.$$
Hence if $n\geq k+k_1$, it follows from (3.1)
$${\overline {I}}^n\subseteq I^{n-k} + {\overline {I}}^{n-k-k_1}Q,$$
and then by (3.2), we have for $n\geq n-2k-k_1$
$${\overline {I}}^n\subseteq I^{n-k} +( { {I}}^{n-2k-k_1}+Q)Q\subseteq I^{n-2k-k_1}+ Q^2. \eqno (3.3)$$

In the following, we will use induction on $i$ to prove
$${\overline {I}}^n \subseteq I^{n-ik-k_1-\cdots-k_{i-1}}+ Q^i.\eqno (3.4)$$
for $i\  (2\leq i \leq n_0+1) $ and for all $n\geq ik +k_1+\cdots+k_{i-1}$.

Clearly, the conclusion holds for $i=2$ by (3.3). Suppose that
$${\overline {I}}^n \subseteq I^{n-jk-k_1-\cdots-k_{j-1}}+ Q^j. $$
for some $j$ with $(2\leq j < n_0)$ and for all $n\geq jk +k_1+\cdots+k_{j-1}$. Then we have
$${\overline {I}}^n \subseteq I^{n-jk-k_1-\cdots-k_{j-1}}+ {\overline I}^{n-jk-k_1-\cdots-k_{j-1}}\cap Q^j.$$
Hence by (3.1), we have for $n\geq jk +k_1+\cdots+k_{j-1}+k_j$
$${\overline {I}}^n \subseteq I^{n-jk-k_1-\cdots-k_{j-1}}+ {\overline I}^{n-jk-k_1-\cdots-k_{j-1}-k_j} Q^j.$$
It follows from (3.2) that for $n\geq (j+1)k +k_1+\cdots+k_{j-1}+k_j$
$${\overline {I}}^n \subseteq I^{n-(j+1)k-k_1-\cdots-k_{j}}+  Q^{j+1}.$$
Hence by induction, (3.4) holds for all $i (2\leq i \leq n_0+1)$. Note that $Q^{n_0+1}=0$.  Thus we have shown
$${\overline {I}}^n \subseteq I^{n-(n_0+1)k-k_1-\cdots-k_{n_0}}.$$
for $n\geq (n_0+1)k+k_1+\cdots+k_{n_0}$ and for all ideals $I$ of $R$. Therefore $i(R)<\infty$, and this ends the  proof of (ii).

\end{proof}

 To give a sufficient condition for  a ring $R$ having the uniform Artin-Rees property, Huneke [Hu] studied two ideals  $T(R)$ and $CM(R)$ of $R$.

For an integer $k$, set $T_k(R)=\underset {I, n}{\bigcap} (I^{n-k}: \overline {I^n})$, where the intersection is taken over all $n$ and all ideals $I$ of $R$. We define $T(R)=\underset {k}{\bigcup} T_k(R)$. It is clear if $k$ increases, the ideals $T_k$ also increase. Hence $T(R)$ is an ideal of $R$. If $T(R)=R$, we say that $R$ has the uniform Brian\c{c}on-Skoda property. Clearly, $R$ has the uniform Brian\c{c}on-Skoda property if and only if there exists an integer $k\geq 0$ such that for $n\geq k$ and all ideals $I$ of $R$
$$\overline{I^n}\subseteq I^{n-k}.$$
By Corollary 2.6, one can see easily $i(R)<\infty$  if $R$ has the uniform Brian\c{c}on-Skoda property.

 We need the following notion  of standard  complex of free modules to describe the  definition of the ideal $CM(R)$. Recalling that  a complex of finitely generated free $R$-modules $ F\textbf{.}$:

              $$  0\rightarrow F_n\overset {f_n}\to  F_{n-1}\rightarrow\cdots\rightarrow
               {F_1}
               \overset {f_1}\to
               \ {F_0} $$
 is said to  satisfy the standard condition on rank if $\text{rank}(f_n) = \text{rank}(F_n)$,
 and $\text{rank}(f_{i+1})+\text{rank}(f_i) = \text{rank}(F_{i})$ for
 $1\leq i < n$. Here we think $f_i$ as given by a matrix  $A_i$, and the
 rank of $f_i$ is the determinantal rank of $A_i$. Let $I(f_i)$ be
 the ideal generated by the rank-size minors of $A_i$.  We say that $F\textbf{.}$ satisfies the standard
condition on  height if $\text{ht}(I(f_i)\geq i $ for all $i$.
Note that the height an ideal $I$ is $\infty$ if $I = R$.
 We define $CM(R)$
to be the ideal generated by all elements $a\in R$ such that for all complexes
$F\textbf{.}$ of finitely generated free $R$-modules satisfying
the standard conditions on height and rank,
$x\text{H}_i(F\textbf{.}) = 0$ for all $i\geq 1$.

As an easy consequence of  Buchsbaum-Eisenbud criterion theorem
for exactness of a free complex [BE], it follows that $CM(R) = R$
if and only if $R$ is a  Cohen-Macaulay ring. One of the important questions concerning
$CM(R)$ is when $CM(R)$ is not contained in any minimal prime
ideal of $R$.  This question is closely related to the existence
of uniform local cohomological annihilators of $R$ which was introduced
by the author [Zh1].

Recall that an element $a\in R$  is said to be a  uniform local cohomological
annihilator of $R$, if

(i) $a$ is not contained  in  any minimal prime ideal of $R$,

(ii) For every maximal ideal $\texttt{m}$, $a$ kills the $i$-th local cohomology module
$\text{H}_\texttt{m}^i(R)$ for $i <\text{ht}(\texttt{m}).$

It is known [cf. Zh2] that $CM(R)$ is not contained in any minimal prime ideal of $R$ if and only if
$R$ has a uniform local cohomological annihilator. Moreover, Zhou proved  a noetherian domain of finite
dimension has   a uniform local cohomological annihilator, provided  $R$ is a homomorphic image of a Cohen-Macaulay ring
[Zh1] or $R$ is an excellent ring [Zh2].

By means of $T(R/P)$ and $CM(R/P)$ for every prime ideal $P$ of $R$, Huneke [Hu]  proved an important  criterion for rings having the uniform Artin-Rees property. The same proof of   Huneke  works for the following criterion, which replaces the condition  $T(R/P)\neq 0$ of [Hu, Theorem 3.4] by $i(R/P)<\infty$.

\begin{theorem} Let $R$ be a  noetherian ring  with infinite residue fields. Assume that for all nonzero prime ideals $P$ of $R$ the following conditions hold:

{\rm{(i)}} $i(R/P)<\infty.$

{\rm{(ii)}} $CM(R/P)\neq 0$.

\noindent Then $R$ has the uniform Artin-Rees property.

\end{theorem}

Combining Theorem 3.2 with  the mentioned results of the author [Zh1] [Zh2] , we have:

\begin{corollary} Let $R$ be a finite dimensional noetherian ring with $i(R/P)<\infty$ for any prime ideal of $R$. If $R$ satisfies one of the following conditions:

{\rm{(i)}} $R$ is an excellent ring.

{\rm{(ii)}} $R$ is a quotient ring of a Cohen-Macaulay ring of finite dimension.

\noindent Then $R$ has the uniform Artin-Rees property.
\end{corollary}

  The following remarkable result is proved by Huneke [Hu, Theorem 4.12], which gives three sufficient conditions for rings having the uniform
 Artin-Rees property. In particular, the theorem asserts that every local ring has   the uniform
 Artin-Rees property.

\begin{theorem}
 Let $R$ be a  noetherian ring of finite  dimension.  $R$ has the uniform Artin-Rees property
 if $R$ satisfies one of the following conditions:

{\rm{(i)}} $R$ is essentially of finite type over a local ring.

{\rm{(ii)}} $R$ is a ring of characteristic $p>0$, and $R$ is module finite over $R^p$.

{\rm{(iii)}} $R$ is essentially of finite type over the ring   of integer numbers $\textbf{Z}$.

\end{theorem}

The another remarkable result of Huneke [Hu] is the following uniform Brian\c{c}on-Skoda  theorem [Hu, Theorem 4.13].

\begin{theorem}
 Let $R$ be a reduced noetherian ring of finite  dimension. $R$ has the uniform Brian\c{c}on-Skoda property if $R$ satisfies one of the following conditions:

{\rm{(i)}} $R$ is essentially of finite type over an excellent local ring.

{\rm{(ii)}} $R$ is a ring of characteristic $p>0$, and $R$ is module finite over $R^p$.

{\rm{(iii)}} $R$ is essentially of finite type over the ring   of integer numbers  $\textbf{Z}$.

\end{theorem}

Now,  we turn to  the main result of the paper. By means of  Proposition 3.1, Theorem 3.4 and Theorem 3.5, we can prove the following main result,  which is a weaker generalization of Theorem 3.5.

\begin{theorem}
 Let $R$ be a  noetherian ring of finite  dimension. Then $i(R)< \infty$
 if $R$ satisfies one of the following conditions:

{\rm{(i)}} $R$ is essentially of finite type over a noetherian local ring.

{\rm{(ii)}} $R$ is a ring of characteristic $p$, and $R$ is module finite over $R^p$.

{\rm{(iii)}} $R$ is essentially of finite type over the ring of integer numbers $\textbf{Z}$.
\end{theorem}

\begin{proof} (i)\  By Corollary 2.3, we may assume $R$ is of finite type over a local ring $(S, \texttt{m})$. Let
$\hat{S}$ be the completion of $S$ in the $\texttt{m}$-adic topology. It is known $\hat {S}$ is faithfully flat over $S$, and thus
$B=R\otimes_S\hat{S}$ is of finite type and faithfully flat over $\hat{S}$.  From Proposition 2.7, it suffices to prove $i(B)<\infty$. Replacing $R$ by $B$, and $S$ by $\hat{S}$, we may assume $R$ is of finite type over an excellent ring $S$.

Write $R=S[X_1, X_2, \cdots, X_r]/I$, where $I$  is  an ideal of the polynomial ring   $S[X_1, X_2, \cdots, X_r]$. Let $P_1, P_2, \cdots, P_t$ be all the minimal prime ideals of $I$. Set
$p_i=P_i\cap S$ for $1\leq i\leq t$. It is not difficult to see that
 $$R_{red}=S[X_1, X_2, \cdots, X_r]/P_1\cap P_2\cap\cdots\cap P _t$$
is of finite type over the reduced excellent ring $S/ p_1\cap p_2\cap\cdots\cap p_t $. By Theorem 3.5, $R_{red}$ has the uniform Brian\c{c}on-Skoda property. In particular $i(R_{red})<\infty$. Now from Theorem 3.4 and Proposition 3.1, we conclude $i(R)<\infty$, and this proves (i).

(ii)\  First,  we  show that $R_{red}$ is module finite over $(R_{red})^p$.
In fact, let $P_1, P_2, \cdots, P_r$ are all the minimal prime ideal of $R$. It is easy to see $R^p/(P_i\cap R^p)$ is naturally isomorphic to $(R/P_i)^p$ and $P_i\cap R^p= {P_i}^p$ is a prime ideal of $R^p$. Hence the nilpotent  ideal of $R^p$
$$I=P_1\cap P_2\cap \cdots \cap P_r\cap R^p = {P_1}^p\cap {P_2}^p\cap \cdots \cap {P_r}^p$$
 is an intersection of prime ideals of $R^p$. Conversely the nilradical ideal J is clearly contained in $I$,  so $I$ is the nilradical ideal of $R^p$. It is easy to see $P_i^p$ are all the distinct   minimal prime ideals of $R^p$.  Moreover, there is a natural isomorphism
 between  the rings $R^p/I$ and $({R_{red}})^p$. So if $R$ is module finite over $R^p$, then
 $R_{red}$ is module finite over $R^p/I$, and thus is module finite over $({R_{red}})^p$.

 Now from Theorem 3.5, $R_{red}$ has the uniform Brian\c{c}on-Skoda property. In  particular $i(R_{red})<\infty$. Now from Theorem 3.4 and Proposition 3.1, we conclude $i(R)<\infty$, and this ends the proof of (ii).

 (iii) \  Clearly, $R_{red}$ is also essentially of finite type over the ring of integer numbers $\textbf{Z}$. Hence from Theorem 3.5  $i(R_{red})<\infty$. The conclusion follows from Theorem 3.4 and Proposition 3.1, and this ends the proof of the theorem.

\end{proof}

An immediate consequence of Theorem 3.5 is the following finiteness property of $i(R)$ for any noetherian local ring $R$.

\begin{corollary}
Let $R$ be a noetherian local ring. Then $i(R)<\infty$.
\end{corollary}

 \vspace* {0.3cm}

\begin{list}
\bf {\textbf{References}}{} \small

\item[\rm{[AW]}] Andersson, M.,   Wulcan, E.: Global effective versions of the Brian\c{c}on-Skoda-Huneke theorem.
{\it Invent. Math. }  {\bf 200}, 607-651(2015)

\item[\rm{[BE]}]   Buchsbaum, D., Eisenbud, D.: What makes a
complexes exact. {\it J. Algebra}   {\bf 25}, 259-268(1973)

\item[\rm{[BS]}] Brian\c{c}on, J.,   Skoda,  H.: Sur la cl\'{o}ture int\'{e}grale d' un id\'{e}al de germes
de fonctions   holomorphes en un point de ${\textbf{C}}^n$.  {\it C.R. Acad. Sci. Paris
S$\acute{e}$r. A.}   {\bf 278}, 949-951(1974)

\item[\rm{[Ei]}] Eisenbud, D.: Commutative algebra. With a view toward algebraic geometry. {\it Graduate
Texts in Mathematics}  {\bf 160}, Springer, New York (1995)

\item[\rm{[Gr]}]  Grothendieck, A.:  Local Cohomology.
 {\it Lect. Notes Math.}  {\bf 41},  Springer-Verlag, New York
 (1963)

\item[\rm{[HK]}]   Herzog, J.,  Kunz,  E.: Der Kanonische Module
eines Cohen-Macaulay Rings. {\it Lecture Notes in Mathematics}
 {\bf 238},  Springer-Verlag, New York(1971)

\item[\rm{[Ho1]}]   Hochster, M.: Contracted ideals from integral extensions of regular rings.
 {\it Nagoya Math. J.}  {\bf 51}  25-43(1973)

\item[\rm{[Ho2]}]   Hochster, M.: Canonical elements in local cohomology modules and the direct
summand conjecture.  {\it  J. Algebra}  {\bf 84},  503-553(1983)

\item[\rm{[HH1]}]  Hochster, M., Huneke C.:  Tight closure,
invariant theory, and the Brian\c{c}on-Skoda   theorem. {\it J. Amer.
Math. Soc.}  {\bf 3}, 31-116(1990)

\item[\rm{[HH2]}]  Hochster, M., Huneke C.:  Infinite integral
extensions and big Cohen-Macaulay algebras. {\it Annals of math.}
{\bf 3}, 53-89(1992)

\item[\rm{[HH3]}] Hochster, M., Huneke C.: Comparison of symbolic
 and ordinary powers of ideals.   {\it Invent. Math.}   {\bf 147},
349-369(2002)

\item[\rm{[HS]}] Huneke, C., Swanson, I.: Integral Closure of Ideals, Rings and Modules. {\it  London Math. Soc.
Lect.  Note Ser.} {\bf 336}, Camb. Univ. Press 2006

\item[\rm{[Hu]}] Huneke, C.: Uniform bounds in noetherian rings.
{\it Invent. Math. }  {\bf 107}, 203-223 (1992)

\item[\rm{[H$\rm {\ddot{u}}$]}] H$\rm {\ddot{u}}$bl, R.:  Derivations and the integral closure of ideals, with an appendix
Zeros of differentials along ideals.  by I. Swanson.  {\it Proc. Amer.
Math. Soc.}  {\bf 127},  3503-3511 (1999)

\item[\rm{[It]}] Itoh, S. Integral closures of ideals of the principal class. {\it Hiroshima
Math. J.}  {\bf 17} 373-375 (1987)

\item[\rm{[Ku]}]  Kunz, E.: On noetherian rings of characteristic p. {\it Am. J. Math.}  {\bf 98}, 999-1013 (1976).

\item[\rm{[LS]}]  Lipman, J., Sathaye, A.: Jacobian ideals and a theorem of Brian\c{c}on-Skoda. {\it Mich.
Math. J.}  {\bf 28}, 199-222 (1981)

\item[\rm{[Ly]}] Lyubeznik, G.: A property of ideals in polynomial rings. {\it Proc. Am. Math. Soc.}
{\bf 98},  399-400 (1986)

\item[\rm{[Ma]}] Matsumura, H.: Commutative Ring Theory, {\it
Cambrige Univ. Press, New York}, (1986)

\item[\rm{[NR]}] Northcott, D., G., Rees, D.: Reductions of ideals in local rings. {\it  Proc.
Cambridge Phil. Soc.} {\bf 50},  145-158 (1954)

\item[\rm{[Re]}] Rees, D.: A note on analytically unramified local rings. {\it J. London
Math. Soc.}  {\bf 36},  24-28 (1961)

\item[\rm{[Sc]}]  Schenzel, P.:   Cohomological annihilators. {\it
Math. Proc. Camb. Phil. Soc.}  {\bf 91}, 345-350 (1982)

\item[\rm{[Sh1]}]  Sharp, R.:  Local cohomology theory in
commutative algebra. {\it Q. J.  Math. Oxford}   {\bf 21},
425-434 (1970)

\item[\rm{[Sw1]}] Swanson, I.:  Joint reductions, tight closure and the Brian\c{c}con-Skoda
Theorem.  {\it J. Algebra} {\bf 147},  128-136 (1992)

\item[\rm{[Sw2]}] Swanson, I.:  Joint reductions, tight closure and the Brian\c{c}con-Skoda
Theorem.  {\it J. Algebra} {\bf 170},  567-583 (1994)

\item[\rm{[Zh1]}]  Zhou, C.:  Uniform annihilators of local
cohomology.  {\it J. Algebra} {\bf 305}, 585-602 (2006)

\item[\rm{[Zh2]}]  Zhou, C.:  Uniform annihilators of local
cohomology of excellent rings. {\it J. Algebra} {\bf 315}, 286-300 (2007)

\end{list}

\end{document}